\documentclass{amsart}

\usepackage{amssymb,amsfonts,amsxtra,amsmath,tikz-cd,xcolor,mathrsfs,
  verbatim,centernot,enumitem,mathtools}
\IfFileExists{quiver.sty}{\usepackage{quiver}}{}
\usepackage{extpfeil}
\usepackage{relsize}
\usepackage[all]{xy}
\usepackage{dsfont}
\usepackage[hidelinks]{hyperref}

\allowdisplaybreaks[2]

\newtheorem{theo}{Theorem}[section]
\newtheorem*{theo*}{Theorem}
\newtheorem{lemm}[theo]{Lemma}
\newtheorem{prop}[theo]{Proposition}
\newtheorem{coro}[theo]{Corollary}

\newtheorem{conj}[theo]{Conjecture}

\theoremstyle{definition}

\newtheorem{defi}[theo]{Definition}
\newtheorem{exam}[theo]{Example}
\theoremstyle{remark}
\newtheorem{rema}[theo]{Remark}
\numberwithin{equation}{section}

\newcommand{\CC}{\mathbb{C}}
\newcommand{\FF}{\mathbb{F}}

\newcommand{\PP}{\mathbb{P}}

\newcommand{\ZZ}{\mathbb{Z}}
\newcommand{\cO}{\mathcal{O}}
\newcommand{\cI}{\mathcal{I}}

\newcommand{\cL}{\mathcal{L}}

\newcommand{\cT}{\mathcal{T}}

\newcommand{\fm}{\mathfrak{m}}

\newcommand{\gr}{\operatorname{gr}}

\newcommand{\rk}{\operatorname{rank}}
\newcommand{\Sing}{\operatorname{Sing}}
\newcommand{\Der}{\operatorname{Der}}

\DeclareMathOperator{\AR}{AR}
\DeclareMathOperator{\md}{md}
\DeclareMathOperator{\diver}{div}

\begin{document}
\title[On Free and Nearly Free Conjecture]
{On Free and Nearly Free Conjecture of Rational Unibranched Projective Plane Curves}

\author{Xiping Zhang}
\address{School of Mathematical Sciences, Key Laboratory of Intelligent Computing and Applications (Ministry of Education),
Tongji University, 
1239 Siping Road, 
Shanghai, China}
\email{xzhmath@gmail.com}
\thanks{}

\date{}

\begin{abstract}
Let $C\subset \PP^2$ be a reduced irreducible complex plane curve with complement $U$. In this paper we prove 
that  $\chi(U)\geq \nu(C)$, where 
$\nu(C)$ is the  freeness defect of $C$ defined in \cite{Dim24}. 

When $C$ is a rational curve with  unibranched singularities, this inequality proves  the conjecture  of Dimca and Sticlaru, that  a  rational  unibranched plane curve is either free
or nearly free.
\end{abstract}
\maketitle

\section{Introduction}

Let $S=\CC[x,y,z]$ be the polynomial ring with its standard degree grading,
and let $C\subset\PP^2$ be a reduced irreducible plane curve cut by $f$
of degree $d\geq 2$. Let $J_f=(f_x,f_y,f_z)$ be the 
Jacobian ideal of $f$. It is natural to ask how the   topology of the   plane curve $C$ is interacted with the algebraic
structure of its Jacobian ideal. 

The free/nearly free conjecture proposed in \cite{DS18a}
 is a particular  instance of this interaction.
Based on the notion of a free curve, Dimca and Sticlaru introduced
nearly free curves by studying the saturation of $J_f$ with respect
to the maximal ideal. For a free curve $J_f$ is
saturated, while for a nearly free curve its saturation quotient is nonzero but each graded piece
has dimension $\leq 1$
(see Definition~\ref{defn; freeandnearlyfree}).
These classes play an important role in the study of global
invariants of plane curves. For example,  
their Milnor algebras $S/J_f$ admit
particularly simple graded resolutions, so their Hilbert series
are determined by the exponents of the curves
(see \cite{DS18a}). In \cite{DS18a}\cite{Dim19} it was shown that their total Tjurina
numbers attain,  or differ by $1$ from the  upper bound $(d-1)^2-\md(f)(d-1-\md(f))$, where $\md(f)$ is the minimal Jacobian syzygy degree.

From a geometric point of view, it is natural to consider the
normalization of $C$ and the local analytic branches at its singular
points. Recall that $C$ is rational if its normalization is isomorphic
to $\PP^1$, and it is unibranched if every singular point has exactly
one local analytic branch.

\begin{conj}[Dimca--Sticlaru]
\label{conj; rationalunibranched}
Every reduced rational   plane curve with only unibranched singularities is either free or nearly free.
\end{conj}

Many cases of this conjecture have already be proved before this present
work. Using Walther's comparison between the Jacobian module and the
cohomology of the Milnor fiber in \cite{Wal17}, Dimca and Sticlaru proved
the conjecture when $d$ is even, when $d$ is a prime power, or when
$\pi_1(\PP^2\setminus C)$ is abelian
(see \cite[Theorem~3.1 and Corollary~3.2]{DS18a}).
They also treated odd degrees under the assumption that no local
monodromy at a cusp has an eigenvalue of order $d$.
Dimca also classified rational unibranched curves of degree
$d\geq6$ with only weighted homogeneous singularities and
showed that they are nearly free (see \cite[Theorem~1.1]{Dim19}).
(see \cite[Theorem~1.1]{Dim19}).  
Further results of Dimca and Sticlaru confirms the conjecture,
in particular, when $d\leq34$ or $\md(f)\leq15$
(see \cite{DS18b}). For a summary of the known results we refer to \cite[\S~5]{Dim24}.

In this paper we provide a uniform proof of Conjecture~\ref{conj; rationalunibranched}.
Let $\nu(C)$ be the freeness defect of $C$ (see Definition~\ref{defn; freedefect}). It is known that $C$ is free if and only if $\nu(C)=0$ and $C$ is nearly free if and only if $\nu(C)=1$. The main result of this paper is:
\begin{theo}
\label{theo; main}
Let $C=V(f)\subset\PP^2_{\CC}$ be a reduced irreducible plane curve
of degree $d\geq2$. Then
\[
\nu(C)\leq\chi(\PP^2\setminus C)=1+2g+\sum_{p\in\Sing C}(b_p-1).
\]
Here $g$ denotes the geometric genus of $C$ and $b_p$ denotes the number of local branches at a singular point $p\in \Sing(C)$.
\end{theo}

In particular, if $C$ is a rational  with unibranched singularities, then  
$\chi(\PP^2 \setminus C)= 1$.     
Consequently  
\begin{coro}
\label{coro; conjectureproved}
Conjecture~\ref{conj; rationalunibranched} is true.
\end{coro}

The theorem also applies to rational nearly unibranched curves, namely
irreducible rational curves whose singularities are unibranched except
for one point with two branches. In this case
$\chi(\PP^2\setminus C)=2$, so $\nu(C)\leq 2$. 
This settles\cite[Conjecture~1.4]{MR4117559}:
\begin{coro} 
Every rational nearly unibranched plane curve $C\subset\PP^2$
satisfies $\nu(C)\leq 2$.
\end{coro}

We briefly sketch the idea of the proof of Theorem~\ref{theo; main}. 
Assuming the existence of a reduced irreducible curve $C$   satisfying 
$\nu(C)>\chi(\PP^2\setminus C)$, we will establish some numerical constraints on the minimal syzygy degree $\md(f)$ and then derive a contradiction. 

We begin in \S\ref{sec; Jacobiansyzygy} by reviewing some basic
properties of Jacobian syzygies. The first numerical constraint
Proposition~\ref{prop; maximumnu} provides an interval of integers $m$ on which
$\dim_\CC N(f)_m$ attains its maximal value $\nu(C)$, under the
assumption $2\md(f)<d$.

In \S\ref{sec; monodromy}, we adapt the monodromy argument of
\cite{DS18b} and
\cite{Wal17} on the global Milnor fiber $F=\{f=1\}$ to obtain the second numerical constraint 
Proposition~\ref{prop; EulercharboundA}. This relates the  
dimensions of $N(f)_{2d-3-q}$ to the dimensions of the eigenspaces
$H^1(F,\CC)_{\lambda_q}$, where  
$\lambda_q=\exp(-2\pi iq/d)$.
This relation then forces $2\md(f)< d$, and therefore Proposition~\ref{prop; maximumnu} applies. Consequently    
\[
\nu(C)\leq
\chi(\PP^2\setminus C)+\dim_\CC H^1(F,\CC)_{\lambda_q}
\text{ for every }\md(f)\leq q\leq d-\md(f).
\]

Since $\nu(C)>\chi(\PP^2\setminus C)$, $H^1(F,\CC)_{\lambda_q}\neq0$ throughout this range.
A detailed analysis of the graded syzygy expression o $H^1(F,\CC)$ then yields, for each such $q$,
a nonzero homogeneous polynomial solution of a differential
equation (see \eqref{eq; ODE}). 

The most essential ingredient in this paper is then Lemma~\ref{lemm; keyobstruction}, 
which rules out the existence of three nonzero homogeneous polynomial solutions of this differential equation whose degrees satisfy
certain numerical constraints.  However, a careful choice of three values of $q$ 
 produces three nonzero solutions whose
degrees satisfy precisely these constraints. 
This then gives the desired contradiction.

\vskip .1in 
\noindent
{\bf Acknowledgement}:
The author is grateful to Zhenjian Wang and Professor Alexandru Dimca for reading the first draft and providing helpful comments.

\vskip .1in
\noindent
{\bf AI Disclosure}: 
This work was originated in purely silicon-based  forms. The author had been using an earlier version of a large language model (LLM) to study the progress on Dimca and Sticlaru's free/nearly free conjecture, and had been discussing the problem with the model.

The next day when the new version of the LLM became available, the author supplied it with a record of the previous discussions and the prompt `keep trying.' The new model then produced a proof of this conjecture. After studying and understanding that proof, the author  unified the  case discussions of the original proof and  reformulated it into the present form. This presented paper is written completely by human  and the author is responsible for the correctness of this paper.

\section{Jaobian Syzygies and the First Numerical Constraints}
\label{sec; Jacobiansyzygy}
Let $S=\CC[x,y,z]=\oplus_{j\geq 0} S_j$ be the coordinate ring of $\CC^3$ graded by homogeneous degree. For any polynomial $f\in S$  we denote by $J_f:=(f_x, f_y, f_z)$ the Jacobian ideal of $f$. Throughout this paper we will assume that $f\in S_d$ is an irreducible homogeneous polynomial of degree $d\geq 2$.

By a Jacobian syzygy of $f$, or a syzygy of the Jacobian ideal $J_f$, we mean a triple $(A, B,C)\in S^3$ such that $Af_x+Bf_y+Cf_z=0$. 
The  set of all Jacobian syzygies of $f$, denoted by $\AR(f)$, is a graded submodule of $S^3$:
\[
 \AR(f)=\bigoplus_{j\geq 0} \AR(f)_m \text{ where } 
 \AR(f)_m:=
 \{(A,B,C)\in S_m^3\vert \ Af_x+Bf_y+Cf_z=0\} \/.
 \]
We will call $\md(f):=\min\{m\vert \AR(f)_m\ne0\}$ the minimal Jacobian syzygy degree of $f$. 

Let $I_f:=(J_f:\fm^\infty)$ be the saturation of $J_f$ with respect to the maximal ideal $\fm=(x,y,z)$. The quotient  $N(f):=I_f/J_f$ is also a graded $S$-module and we write $N(f)_m$ for its degree $m$ piece.  

\begin{defi}[{\cite[Definition~2.4]{DS18a}}]
\label{defn; freeandnearlyfree}
Let $C=V(f)\subset \PP^2$ be a reduced irreducible plane curve.
\begin{enumerate}
	\item We say 	$C$ is a free curve if $\AR(f)$ is a free $S$-module of rank two.
    \item We say $C$ is nearly free if $N(f)\neq 0$ and 
$\dim_\CC N(f)_m\leq 1$ for every $m$.
\end{enumerate} 
\end{defi}
\begin{rema}
We note that the word `free' may cause some confusion across references. Here by a free curve we actually mean that the affine cone $\tilde{C}$ is a free divisor in $\CC^3$ in Saito's sense \cite{KSaito80}. The projective curve $C$ is always a free divisor, though it may not be a free curve.
\end{rema}

Following \cite[\S 2]{Dim24}, we make the following definition.
\begin{defi}
\label{defn; freedefect}
We call  $\nu(C):=\max \{ \dim_{\CC}N(f)_m\ \vert \  m\in\ZZ \}$  the
 freeness defect of $C$ (or of $f$).
\end{defi}
The module $\AR(f)$ is free if and only if $J_f$ is $\fm$-saturated, i.e., $N(f)=0$ (see \cite[Proposition 2.3]{DS18a}). On the other hand, 
$C$ is nearly free if and only if    $N(f)\neq 0$ and $\nu(C)\leq 1$. Thus we have the following numerical criterion:
\begin{equation}
\label{eq; criterionfreeandnearlyfree}
C\text{ is free}\iff\nu(C)=0,
 \qquad
 C\text{ is nearly free}\iff\nu(C)=1 \/.	
\end{equation}

The first numerical constraint comes from  \cite[Corollary~4.3]{DP16}, which states that for a generic linear form $l\in S_1$, the left multiplication 
\[
L_l\colon N(f)_m\longrightarrow N(f)_{m+1}\/,
\quad  [g]\mapsto [l\cdot g]
\]
is injective for $m< \frac{3d-6}{2}$ and surjective for $m\geq \left\lfloor \frac{3d-6}{2} \right\rfloor$. Consequently we have 
\begin{equation}
\label{eq; centralmaximum}
 \nu(C)=\dim_{\CC} N(f)_{k_0} \text{ for } k_0: =\left\lfloor\frac{3d-6}{2}\right\rfloor.
\end{equation}

Now we fix a  nonzero homogeneous Jacobian syzygy 
$\rho\in \AR(f)_{\md(f)}$ of minimal  degree $\md(f)$. We recall from \cite[Thm. 3.2]{DS20} the following proposition. 
\begin{prop}
\label{prop; maximumnu}
All the homogeneous Jacobian syzygies of low degrees are multiples of $\rho$:  
\begin{equation}
\label{eq; minimalsyzygy}
\AR(f)_m=S_{m-\md(f)}\cdot \rho \text{ for any } m<d-\md(f)-1 \/.
\end{equation}
If moreover we assume that $2\md(f)<d$, then  for any  $d+\md(f)-3\leq m \leq 2d-\md(f)-3$,  
\begin{equation}
\label{eq; maximalrangenu}
\nu(C)=\dim_{\CC} N(f)_{m}  \/.
\end{equation}
\end{prop}
   
\begin{proof}
Write $\rho=(A,B,C)$. The minimality of $\md(f)$ implies
that $\gcd(A,B,C)=1$.

Let $K=\textup{Frac}(S)=\CC(x,y,z)$ be the field of rational functions. For any homogeneous syzygy $\rho'=(A',B',C')\in\AR(f)_{m}$, since 
$\rho\cdot \nabla_f=\rho'\cdot \nabla_f=0$ we have
$\rho\times \rho'=\frac{P}{Q} \nabla_f$ for some $P,Q\in S$ such that $\gcd (P, Q)=1$. Then $Q$ necessarily divides $\nabla_f$. However we have assumed $f$ to be irreducible, thus $\gcd(f_x,f_y,f_z)=1$ and therefore $Q\in \CC^*$.

Comparing degrees we have $\deg P=\md(f)+m-d+1$. Thus if $m< d-1-\md(f)$ then $P=0$ and $\rho'=\frac{P'}{Q'}\cdot \rho$. Since  $\gcd(A, B,C)=1$, we have $Q'\in \CC^*$ and   $\rho'\in S\cdot \rho$. This proves \eqref{eq; minimalsyzygy}, which was also used
in \cite[Proposition~3.4]{DS18b}.

For the second statement we adapt the Bourbaki scheme argument from \cite{DP16} and \cite{JNS24}.
Recall that  $\widetilde{\AR(f)}$ is the twisted sheaf of logarithmic derivations $\Der_{\PP^2}(-\log C)(-1)$. This is a locally free $\cO_{\PP^2}$-sheaf of  rank $2$. Then  by \cite[\S~3]{DS14} and \cite[Theorem 4.1]{DP16} we have
\[
\AR(f)_m=H^0\bigl(\PP^2, \Der_{\PP^2}(-\log C)(m-1)\bigr)\/; \quad 
N(f)_m= H^1 \bigl(\PP^2,\Der_{\PP^2}(-\log C)(m-d)\bigr) \/.
\]

The minimal syzygy $\rho$ induces a section $\tilde{\rho}\in H^0\bigl(\PP^2, \Der_{\PP^2}(-\log C)(\md(f)-1)\bigr)$ whose zero subscheme $Z$ has dimension $\leq 0$. This is the Bourbaki scheme of $\rho$.  
Recall that   $H^1\bigl(\PP^2, \cO(k)\bigr)=0$ for any 
$k\in \ZZ$ and $H^2\bigl(\PP^2, \cO(k)\bigr)=H^0\big(\PP^2, \cO(-k-3)\bigr)^\vee=0$ 
for any $k\geq -2$ by Serre duality. Thus when $m\geq d+\md(f)-3$ we have
 \[
 H^1\bigl(\PP^2, \cO(m-d+1-\md(f))\bigr)
 =H^2\bigl(\PP^2, \cO(m-d+1-\md(f))\bigr)=0 \/.
 \] 
From the canonical short exact sequence
\[
 0\longrightarrow\cO_{\PP^2}(1-\md(f))\longrightarrow
 \Der_{\PP^2}(-\log C)
 \longrightarrow\cI_Z(\md(f)-d+2)\longrightarrow0 
\]
we have, when $d+\md(f)-3\leq m$ ,  
\[
 H^1\bigl(\PP^2, \Der_{\PP^2}(-\log C)(m-d)\bigr)\simeq H^1\bigl(\PP^2, \cI_Z(m+\md(f)-2d+2)\bigr) \/.
\]
By the standard exact sequence of the   subscheme $Z$, where we set $m'=m+\md(f)-2d+2$
\[
 0\longrightarrow\cI_Z(m')\longrightarrow
 \cO_{\PP^2}(m')
 \longrightarrow\cO_Z(m')\longrightarrow0 
\]
we have, when $m'<0$, i.e., $m \leq 2d-\md(f)-3$,  
\[
H^1\bigl(\PP^2, \cI_Z(m')\bigr)\simeq 
H^0\bigl(Z, \cO_Z(m') \bigr)
\simeq H^0\bigl(Z, \cO_Z \bigr) \text{ is independent of } m \/.
\]

Finally, under the assumption $2\md(f)\leq d$ we have the interval 
$[d+\md(f)-3,\  2d-\md(f)-3]$ 
which is centered at $(3d-6)/2$. In particular this interval contains $k_0$. This completes the proof.
\end{proof}

\section{Monodromy and the Second Numerical Constraints}
\label{sec; monodromy}
Since $f$ is homogeneous of degree $d$, its local Milnor fibration at the origin extends to the global fibration $f\colon \CC^3\setminus V(f)\longrightarrow \CC^*$. We call $F:=f^{-1}(1)$ the global Milnor fiber of $f$.
By \cite[Chapter~3, \S1]{DimcabookSingofHypersurface}, for $0<|\eta|\ll \epsilon\ll 1$ sufficiently small, we have a diffeomorphism 
\[
F_{f,0}:=B^\circ_\epsilon(0)\cap f^{-1}(\eta)\cong F
\]
between the local and global Milnor fibers. 

Let $\pi\colon F\rightarrow U:=\PP^2\setminus C$  be the restriction of the quotient map
$\CC^3\setminus\{0\}\rightarrow \PP^2$. This is a $d$-fold cyclic covering map with 
geometric monodromy action  
\[
T\colon F \longrightarrow F\/; \quad (x,y,z)\mapsto \exp(2\pi i/d)\cdot (x,y,z)
\]
Let $T_i:=T^*\colon H^i(F,\CC)\to H^i(F,\CC)$ denote the
 cohomological monodromy action. For each $0\le q<d$ we set   
 $\lambda_q=\exp(-2\pi iq/d)$ and denote the  corresponding eigenspace by
 \[
H^i(F,\CC)_{\lambda_q}:=\ker \left[ T_i-\lambda_q \textup{Id}\colon H^i(F, \CC)\to H^i(F, \CC)\right] \/.
\]
 
Since $T_i^d=\operatorname{Id}$, the operator $T_i$ is semisimple,
and hence
\[
H^i(F,\CC)
=
\bigoplus_{q=0}^{d-1} H^i(F,\CC)_{\lambda_q},
\quad
\dim_\CC H^i(F,\CC)_{\lambda_q}=a^i_q,
\]
where  $a^i_q$ denotes the
multiplicity of $\lambda_q$ in the characteristic
polynomial of $T_i$. When $i=1$, the characteristic
polynomial of $T_1$ is called the Alexander polynomial of $C$ and is denoted by $\Delta_C(t)$.

 Since $\pi$ is a finite covering,
$R^j\pi_*\CC_F=0$ for $j>0$ and $\pi_*\CC_F$ is a local
system of rank $d$. Define
\[
\cL_q
:=
\ker\bigl(
\cT-\lambda_q\operatorname{Id}
\colon \pi_*\CC_F\longrightarrow\pi_*\CC_F
\bigr) \/,
\]
where $\cT$ is  induced  by the geometric monodromy $T$. Each $\cL_q$ is a rank-one local system on $U$ 
and we have a decomposition
\[
\mathrm{R}\pi_*\CC_F
\simeq
\pi_*\CC_F
\simeq
\bigoplus_{q=0}^{d-1}\cL_q.
\]
Taking cohomology we obtain natural isomorphisms (see \cite[Proposition~6.4.6]{DimcabookSheavesinTopology}):
\[
H^i(F,\CC)_{\lambda_q}
\simeq
H^i(U,\cL_q),
\qquad 0\leq q\leq d-1.
\] 

The second numerical constraints consist of the following properties. 
\begin{prop}
\label{prop; EulercharboundA}
For every $1\leq q< d$ we have 
\begin{equation}
\label{eq; EulercharboundA}
 \dim_\CC N(f)_{2d-3-q}\leq \chi(U)+ \dim_\CC H^1(F, \CC)_{\lambda_q}.
\end{equation}
\end{prop}

\begin{proof}
Since $F$ is a smooth complex affine variety, the cohomology group
$H^i(F,\CC)$ carries Deligne's canonical mixed Hodge structure
(see \cite{Del71}\cite{Del74}). Moreover the  monodromy action $T$   on $H^i(F,\CC)$ is
compatible with  the Hodge
filtration $F^\bullet H^i(F,\CC)$ (see \cite[\S4, Steps~3--4]{Wal17}).
Let $F^\bullet H^i(F,\CC)_{\lambda_q}$ be the induced Hodge filtration on each eigenspaces.
 
From \cite[Theorem~4.3]{Wal17} we have
\[
\dim_\CC N(f)_{2d-3-q}
 \leq \dim \textup{Gr}_{F^\bullet}^1 H^2(F,\CC)_{\lambda_q}:=\dim_\CC  \frac{F^1 H^2(F,\CC)_{\lambda_q}}{F^2 H^2(F,\CC)_{\lambda_q}}
 \/.
 \]

Since $U$ is  complex algebraic, for  the local system $\cL_q$ we   have
\[
\chi(U, \cL_q)=\rk (\cL_q)\cdot \chi(U, \CC_U)=\chi(U) \/.
\] 

Since $F$ is connected, $H^0(U,\cL_0)\simeq H^0(F,\CC)\simeq\CC$ and thus $H^0(U,\cL_q)=0$ for
$1\leq q<d$. Then
\[
 \chi(U)=\dim_\CC H^2(U, \cL_q)-\dim_\CC H^1(U, \cL_q)
 =\dim_\CC H^2(F, \CC)_{\lambda_q}-\dim_\CC H^1(F, \CC)_{\lambda_q}.
\]
Combine with the previous inequality we then have
\[
\dim_\CC N(f)_{2d-3-q}
\leq \dim\gr^{F^\bullet}_1 H^2(F,\CC)_{\lambda_q}
\leq \dim_\CC H^2(F,\CC)_{\lambda_q}=\chi(U)+ \dim_\CC H^1(F, \CC)_{\lambda_q} \/. 
\qedhere
\]
 \end{proof}

We have  the following  Zariski type result from \cite[Proposition 2.1]{AD15}.
\begin{prop} 
\label{prop; primeordervanishing}
Assume that $C=V(f)\subset \PP^2$ is an irreducible reduced plane curve. If $\frac{d}{\gcd (d, q)}=p^k$ for some prime number $p$ and $k\geq 1$, i.e., $\lambda_q$ has order $p^k$, then $H^1(F,\CC)_{\lambda_q}= 0$. 
\end{prop}

The next ingredient is a description of the eigenspaces
$H^1(F,\CC)_{\lambda_q}$ in terms of syzygies, following
\cite[\S~3.1]{DS18b} and \cite[\S~2]{DS20}. 
For completeness, and to clarify the   degree constraints
involved in the cited results, we include     necessary details for the proof of the next proposition.

For any integer $m$ we consider the homogeneous syzygies that have divergence $0$:
\begin{equation}
\label{eq; Km}
 K_m(f):=\left\lbrace \theta=(A,B,C)\in\AR(f)_m \ \vert \ \textup{div}(\theta):=A_x+B_y+C_z=0
 \right\rbrace \/.
\end{equation} 
 
\begin{prop}
\label{prop; H1isdivzero}
When $C$ is a reduced curve, for any  $1\le q<d$ we have
\begin{equation}\label{eq; H1}
\dim_\CC H^1(F, \CC)_{\lambda_q}=\dim K_{q-2}(f)+\dim K_{d-q-2}(f).
\end{equation}	
Consequently, when $C$ is irreducible, for any $1\le q<d$ we have
\begin{equation}
\label{eq; boundH1byK}
 \dim_\CC N(f)_{2d-3-q}\leq \chi(U)+\dim K_{q-2}(f)+\dim K_{d-q-2}(f) \/.
\end{equation}
\end{prop}
\begin{proof}
Let $\Omega_{S/\CC}^j$ be the module of differential $j$-forms on $S=\CC[x,y,z]$ with natural grading 
\[
\deg x=\deg y=\deg z=\deg \mathrm{d}x=\deg \mathrm{d}y=\deg \mathrm{d}z=1.
\]
The usual differentiation $\mathrm{d}$ is  degree-preserving, while taking wedge with  $\mathrm{d}f$ increases degree by $d$. 
Let $K_f^\bullet:=(\Omega_{S/\CC}^\bullet, \ \mathrm{d}f\wedge)$ be  the Koszul complex of $J_f$. Its cohomology modules $H^{k}(K_f^\bullet)$ are also graded and we write $H^{k}(K_f^\bullet)_{m}$ for the degree $m$ piece. 

Following  \cite[(2.7)]{DS20}, the  pole order filtration induces the first page of  
spectral sequence as  
\[
E_1^{s,t}(f)_m :=
H^{s+t+1}(K_f^\bullet)_{td+m} \text{ and }  
\mathrm{d}_1\colon E_1^{s,t}(f)_m
\longrightarrow E_1^{s+1,t}(f)_m\/, \   [\omega]\mapsto [\mathrm{d}\omega] \/.
\]

We know that $\dim_\CC H^1(F, \CC)_{\lambda_q}$ equals the multiplicity $m(\lambda_q)$ of $\lambda_q$ in the Alexander polynomial $\Delta_C(t)$. By \cite[Corollary 2.4]{DS20} we have 
\[
 \dim_\CC H^1(F, \CC)_{\lambda_q}=
 \dim E_2^{1,0}(f)_q+\dim E_2^{1,0}(f)_{d-q} \/.
\]

Since $f$ is reduced, $\gcd (f_x, f_y, f_z)=1$ and hence $H^1(K_f^\bullet)=0$. 
Consequently 
\[
E_2^{1,0}(f)_m
=
\ker\left[
\mathrm{d}_1\colon 
H^2(K_f^\bullet)_m
\longrightarrow
H^3(K_f^\bullet)_m
\right].
\]

Since $(\Omega_{S/\CC}^1)_{m-d}=0$ for any $1\leq m <d=\deg(f)$, we see that 
\[
 H^2(K_f^\bullet)_m
=\frac{\ker [\mathrm{d}f\wedge\colon (\Omega_{S/\CC}^2)_m\to  (\Omega_{S/\CC}^3)_{m+d}]}{\mathrm{d}f\wedge (\Omega_{S/\CC}^1)_{m-d}}
=
\{\omega\in(\Omega^2_{S/\CC})_m\ \vert \  
  \mathrm{d}f\wedge\omega=0\}  \/.
\]
Then by \cite[Remark~3.1]{DS19} there is an isomorphism 
$\varphi_m\colon \AR(f)_{m-2}\xrightarrow{\simeq} H^2(K_f^\bullet)_m$  for $1\leq m <d$. 

Each syzygy $\theta=(A,B,C)$ uniquely determines
a vector field $D_\theta$ and a polynomial $2$-form
$\omega_\theta$:
\[
D_\theta:= A \,\partial_x + B \,\partial_y +C\,\partial_z \/, \quad 
\omega_\theta
:=
A\,\mathrm{d}y\wedge \mathrm{d}z+B\,\mathrm{d}z\wedge \mathrm{d}x+C\,\mathrm{d}x\wedge \mathrm{d}y \/.
\] 
An easy computation shows that 
$\mathrm{d}f\wedge\omega_\theta=(Af_x+Bf_y+Cf_z)\cdot \mathrm{d} x\wedge \mathrm{d} y\wedge \mathrm{d}z$, thus 
$\mathrm{d}f\wedge \omega_\theta=0$ for every $\theta\in \AR(f)_{m-2}$. 
The isomorphism $\varphi_m$ simply sends $\theta$ to $\omega_\theta$.   
 
Since  $\mathrm{d}\omega_\theta
=\textup{div}(\theta)\,dx\wedge dy\wedge dz$,
under the identification  
$\varphi_m$  the differential $d_1$ becomes
\[
\textup{div}\colon 
\AR(f)_{m-2}\longrightarrow S_{m-3},
\quad
\theta=(A,B,C)\longmapsto \textup{div}(\theta)=A_x+B_y+C_z.
\]
Consequently   
$E_2^{1,0}(f)_m
\cong
\ker\left(
\textup{div}
\right)
=
K_{m-2}(f).
$ for each $1\leq m<d$.

Combining with Proposition~\ref{prop; EulercharboundA}  we obtain the other statement.
\end{proof}

\section{Key Obstruction: the Third Numerical Constraints}
\label{sec; keyobs}
In this section we prove 
the key obstruction  that will constrain the possible values of the eigenvalues $\lambda_q$, for any possible counter-examples of Theorem~\ref{theo; main}. Again we assume that $C=V(f)\subset \PP^2$ is an irreducible reduced plane curve and $\deg f=d\geq 2$.

\begin{lemm}
\label{lemm; linearobstruction} 
Let $\ell\in S_1$ be a nonzero linear form.  
If $m<d-1$,   then $\AR(\ell)_m\cap K_m(f)=\{0\}$.
\end{lemm}

\begin{proof}
Suppose that $m<d-1$ and $\theta=(A, B,C) \in \AR(\ell)_m\cap K_m(f)$, that is to say, 
\[
 Af_x+Bf_y+Cf_z=A\ell_x+B\ell_y+C\ell_z=0 
 \text{ and } \diver(\theta)=A_x+B_y+C_z=0.
\]
 
Since linear coordinate change multiplies a nonzero constant to the first derivation relations, up
 to  linear coordinate changes we may assume  $\ell=z$ and $[0:0:1]\in C$. 
Then the relations become
\[
C=0, \quad  Af_x+Bf_y=0, \quad A_x+B_y=0.
\]
Restricting to the affine complement $U_z:=\PP^2\setminus V(\ell)$
we set
\[
 a:=A(x,y,1),\quad b:=B(x,y,1),\quad g:=f(x,y,1).
\]
Since $f$ is irreducible, $z$ does not divide $f$ and hence $g$ also has degree $d$.   Then we have 
\[
a_x+b_y=0 \/, \quad ag_x+bg_y=0
\]

Define $1$-form $\omega:=-b \mathrm{d}x+a \mathrm{d}y$, then $\omega$ is closed since  $\mathrm{d} \omega = (a_x+b_y) \mathrm{d}x\wedge \mathrm{d}y =0$. By Poincar\'e lemma $\omega$ is exact and there exists a polynomial $H\in \CC[x,y]$ such that $H_x=-b$ and $H_y=a$. 

The vector field  
$\nu:=a\partial_x+b\partial_y$ is then tangent to both curves $C=V(g)$ and $C'=V(H-H(0))$, therefore $C$ and $C'$ share a common irreducible component passing through $0$. Since $g$ is irreducible, this shows that $g$  divides $H-H(0)$. 
However,   $\theta\in S_m^3$ implies that $\deg H\leq m+1<d=\deg g$. Thus we have $a= b\equiv 0$, and consequently we have $A=B=C=0$.
\end{proof}

\begin{exam}
The assumption $m<d-1$ is necessary. For every $d$ we set $\ell=z$, then the polynomial $f=x^d+y^{d-1}z$  is irreducible  and  we have $(f_y, -f_x, 0)\neq 0 \in \AR(\ell)_{d-1}\cap K_{d-1}(f)$.
\end{exam}

We now give the essential obstruction used in the proof of Theorem~\ref{theo; main}.  
\begin{lemm}
\label{lemm; keyobstruction}
Given any  nonzero Jacobian syzygy of  minimal  degree $\rho \in\AR(f)_{\md(f)}$ and any pair of   integers $(L, \alpha)$  satisfying 
\begin{equation}
\label{eq; obslemmanumassum}
 1\leq L \/, \quad 0\leq 2\alpha<L,\quad \alpha+2L-3<\md(f) ,\quad \md(f)+\alpha<d-1.
\end{equation}
There do not exist nonzero polynomial ${P, Q, R}\subset S$ such that  
\[
P\in S_\alpha \/, \quad Q\in S_{L-1}\/, \quad 
R\in S_L \text{ and }
\diver(P\cdot \rho)=\diver(Q\cdot \rho)=\diver(R\cdot \rho)=0.
\]
\end{lemm}

\begin{proof}
We prove by contradiction.
Assume that there exists a nonzero Jacobian syzygy $\rho=(A,B,C)$ and nonzero polynomials $\{P, Q, R\}$ satisfying all the hypotheses. 
Let $D_\rho:=A\partial_x+B\partial_y +C\partial_z$ be the assigned vector field of $\rho$. For any $g\in S$ we have  
\[
\diver(g\cdot \rho) =(gA)_x+ (gB)_y+(gC)_z=
 D_\rho(g) + \diver(\rho)\cdot g \/.
\]
Thus by assumption we have 
\begin{equation}
\label{eq; ODE}
D_\rho (P)+\diver(\rho) P =
D_\rho (Q)+\diver(\rho) Q =
D_\rho (R)+\diver(\rho) R =0 \/.
\end{equation}

Set $u=Q/P$ and $v=R/P$, a direct computation gives
\[
D_\rho(u)=D_\rho (v)=D_\rho (f)=0 \/.
\]

Let $\FF:=\textup{Frac}(S)=\CC(x, y, z)$ be the field of rational functions. 
Since $\rho\neq 0$,  $\mathrm{d}u$, $\mathrm{d}v$ and $\mathrm{d}f$ are linearly dependent vectors in $\FF^3$ and hence $\mathrm{d}f\wedge \mathrm{d}u\wedge \mathrm{d}v=0$. 
Since 
\[
\mathrm{d}u \wedge \mathrm{d}v= \frac{
P\mathrm{d}Q\wedge \mathrm{d}R+
Q \mathrm{d} R \wedge \mathrm{d} P+
R \mathrm{d}P\wedge \mathrm{d} Q}{P^3}
=: \frac{\Omega}{P^3} \/, 
\]
The polynomial $2$-form $\Omega$ uniquely determines a vector field $\Xi$ by 
$\iota_\Xi(\mathrm{d}x\wedge\mathrm{d}y\wedge\mathrm{d}z)=\Omega$, where
$\iota_\Xi$ denotes contraction by $\Xi$. 
Its coefficients are homogeneous of degree
$\alpha+2L-3$. 
Since $\mathrm{d}f\wedge \mathrm{d}u\wedge \mathrm{d}v=0$, we have $\Xi(f)=0$ and hence $\Xi\in \AR(f)_{\alpha+2L-3}$.

If $\alpha+2L-3<0$, then $\Xi=0$. Otherwise, $\Xi$ is a
Jacobian syzygy of degree less than $\md(f)$, so the
minimality of $\md(f)$  implies that $\Xi=0$ and therefore $\mathrm{d}u\wedge \mathrm{d}v=0$.

Set $k=L-\alpha$ for simplicity. Then \eqref{eq; obslemmanumassum} implies that $k\ge1$ and $\alpha<k$.
Let $E:=x\partial_x+y\partial_y+z\partial_z$ be the Euler vector field. Then $E(u)=(k-1)u$ and $E(v)=kv$ since $u$ and $v$ are homogeneous of degrees $k-1$ and $k$, respectively. 
Then  we have 
\begin{align*}
\mathrm{d}\left(\frac{v^{k-1}}{u^k}\right)
&=\frac{v^{k-2}}{u^{k+1}}\cdot 
  \left((k-1)u\,\mathrm{d}v-kv\,\mathrm{d}u\right) \\
&=\frac{v^{k-2}}{u^{k+1}}\cdot 
  \left(E(u)\,\mathrm{d}v-E(v)\,\mathrm{d}u\right) \\
&=\frac{v^{k-2}}{u^{k+1}}\cdot 
  \iota_E \left(\mathrm{d}u\wedge\mathrm{d}v\right)
=0.
\end{align*}
Here $\iota_E$ denotes  contraction by $E$. 
Thus $v^{k-1}= C_0 \cdot u^k$ for some constant $C_0\in\CC^*$.   
Clearing denominators we then have $PR^{k-1}=C_0 \cdot Q^k$.

Write $h=\frac{Q}{\gcd (Q, R)}$. Then  $h^k$ divides $P$.
However $\deg P=\alpha<k$, thus $h$ is a constant and  $Q$ divides $R$. 
Consequently $\ell:=\frac{R}{Q}\in S_1$ is a   nonzero linear
form. 

Finally we note that $D_{P\cdot \rho}(\ell)=P\cdot D_\rho(\ell)=0$ and thus 
$P\cdot \rho\neq 0\in \AR(\ell)_{\md(f)+\alpha}\cap K_{\md(f)+\alpha}(f)$. Since $\md(f)+\alpha<d-1$, this contradicts Lemma~\ref{lemm; linearobstruction}. 
\end{proof}

\section{The Proof of Theorem~\ref{theo; main}}
\label{proof}
Now we prove Theorem~\ref{theo; main} by contradiction. 
Let $C\subset \PP^2$ be an irreducible reduced curve  cut by a homogeneous polynomial $f$ of degree $\geq 2$. We assume that $\nu(C)> \chi(U)$. 

Let $\md(f)=\min  
\{m\vert \AR(f)_m\neq 0\}$ be the minimal Jacobian syzygy degree. We fix a minimal degree Jacobian syzygy $\rho=(A,B,C)\in \AR(f)_{\md(f)}$.
We will use the  numerical constraints discussed previously to derive a contradiction.

Recall from Equation~\eqref{eq; centralmaximum} that
$\nu(C)=\dim N(f)_{k_0}$ for  $k_0=\lfloor \frac{3d-6}{2} \rfloor$. 
The observation is that  $k_0=2d-3-c$, where $c:=\lceil \frac{d}{2} \rceil$. Clearly  $1\leq c <d$ and $d\leq 2c\leq d+1$. By Proposition~\ref{prop; EulercharboundA}   we have 
\[
\nu(C)= \dim_\CC N(f)_{2d-3-c} \leq \chi(U) +\dim_\CC H^1(F, \CC)_{\lambda_c} \/.
\] 
Combining this inequality with 
Proposition~\ref{prop; H1isdivzero}
and the  
$\nu(C)> \chi(U)$ assumption, 
we have
\[
0< \dim_\CC H^1(F, \CC)_{\lambda_c} =
\dim_\CC K_{c-2}(f) + \dim_\CC K_{d-c-2}(f) \/.
\]
Since $K_m(f)\subset \AR(f)_m$ and $d-c\leq c \leq d-c+1$, the minimality of $\md(f)$ implies that 
\begin{equation}
\label{eq; c1}
\md(f) \leq c-2 < \frac{d}{2} \/.
\end{equation}

Therefore by Equation~\eqref{eq; maximalrangenu} and  Inequality ~\eqref{eq; boundH1byK}, for every $\md(f)\leq q \leq d-\md(f)$ we have 
\begin{equation}
\label{eq; c0}
\nu(C) = \dim_\CC N(f)_{2d-3-q} \leq \chi(U)+ \dim_\CC K_{q-2}(f) + \dim_\CC K_{d-q-2}(f)  \/.
\end{equation}

Since $\md(f)\leq q \leq d-\md(f)$, both $q-2$ and $d-q-2$ lie in  the range of Equation~\eqref{eq; minimalsyzygy}. Set
\[
 W_m:=\left\lbrace g\in S_m \ \vert \  g\cdot \rho\in K_{m+\md(f)}(f) \right\rbrace
 =\left\lbrace g\in S_m\ \vert \ \diver(g\cdot \rho)=0  \right\rbrace  \/,
 \] 
 then  
$K_{q-2}(f)= W_{q-\md(f)-2}\cdot \rho$ and  
$K_{d-q-2}(f)= W_{d-q-\md(f)-2}\cdot \rho$. 
Assuming $\nu(C)> \chi(U)$, 	
\begin{equation}
\label{eq; c2} 
\text{ either } W_{q-\md(f)-2} \neq 0 \text{ or } W_{d-q-\md(f)-2}\neq 0 \text{ for every }
\md(f)\leq q \leq d-\md(f)  \/.	
\end{equation} 

\begin{lemm}
\label{lemm; lowerboundmd(f)}
We have $\md(f) > \frac{3(d-1)}{7}$ under the assumption that $\nu(C)> \chi(U)$.
\end{lemm}
We proceed with the proof and postpone the proof of the lemma.

Set $L:=d-2\md(f)-2$.
By Inequality \eqref{eq; c1} we have   $\{\md(f), \md(f)+1, c\}\subset [\md(f), d-\md(f)]$ and $L\geq 1$.
  Substituting $q=\md(f)$, $q=\md(f)+1$ and $q=c$ into  \eqref{eq; c2} we have 
\begin{equation}
\label{eq; c3}
W_{L}\neq 0,\quad W_{L-1} \neq 0, 
\quad W_{c-\md(f)-2}\neq 0 \text{ or } 
W_{d-c-\md(f)-2}\neq 0 
\/.
\end{equation}
 
Since  $c-\md(f)-2\geq 0$, we 
 may choose an integer $\alpha\in\{c-\md(f)-2,\ d-c-\md(f)-2\}$ such that $\alpha\geq 0$ and 
$W_\alpha\neq 0$. 
Since $2c\geq d$, we have $d-c-\md(f)-2\leq c-\md(f)-2$ and therefore
\[
\max \left\lbrace0, \frac{L-3}{2} \right\rbrace \leq \alpha\leq c-\md(f) -2 \leq \frac{L-1}{2} \/.
\]
Then \eqref{eq; c3} shows that there exist nonzero polynomials $P\in S_\alpha$, $Q\in S_{L-1}$, $R\in S_L$ such that 
\[
\textup{div}(P\cdot \rho)=\textup{div}(Q\cdot \rho)=\textup{div}(R\cdot \rho) =0 
\]
and
\[
2\alpha <L\/, \quad \alpha+\md(f)\leq c-2<d-1\/, \quad \alpha+2L-3
 \leq \frac{5L-7}{2}
   =\frac{5d-10\md(f)-17}{2}<\md(f) \/.
\]
The last inequality is due to Lemma~\ref{lemm; lowerboundmd(f)}, which says  
$12\md(f)-5d+17>
 \frac{36(d-1)}7-5d+17 >0$.

But this contradicts Lemma~\ref{lemm; keyobstruction},  therefore the assumption 
$\nu(C) > \chi(U)$ is impossible  and thus   Theorem~\ref{theo; main} is true.

Let $g$ be the genus of the normalization  of $C$ and
let $b_p$ be the number of local branches at a singular point $p\in C$.
Since the normalization lifts $p$ to $b_p$ distinct points,  we have
\[
\chi(C)=2-2g-\sum_{p\in\Sing C}(b_p-1),
\quad
\chi(\PP^2\setminus C)
=1+2g+\sum_{p\in\Sing C}(b_p-1).
\]

It remains to prove Lemma~\ref{lemm; lowerboundmd(f)}. The proof  relies on 
Proposition~\ref{prop; primeordervanishing}, which says that the  eigenspaces $H^1(F, \CC)_{\lambda_q}$ vanish whenever the eigenvalue $\lambda_q$ has order of a prime power.   Set
\begin{equation}\label{eq; beta}
 \beta(d):=\max\left\lbrace\min \{q, d-q\}\Big\vert
   1\leq q<d \text{ and } \frac d{\gcd(d,q)}=p^e
                 \text{ for some prime }p 
                 \right\rbrace \leq c=\lceil \frac{d}{2}\rceil .
\end{equation} 

If $\md(f)\leq\beta(d)$, then there exists an integer $m$
with $\md(f)\leq m\leq d-\md(f)$ such that $\lambda_m$
has prime-power order.    
Then  $H^1(F, \CC)_{\lambda_m}=0$ by Proposition~\ref{prop; primeordervanishing}
 and hence $K_{m-2}(f)=K_{d-m-2}(f)=0$. 
However, the assumption $\nu(C)>\chi(U)$ and the Inequality \eqref{eq; c0} shows that  $H^1(F, \CC)_{\lambda_m}\neq 0$ for any $m\in [\md(f), d-\md(f)]$. 
This contradiction  forces  $\md(f)> \beta(d)$. 

It then suffices to prove $\beta(d)\geq \frac{3(d-1)}{7}$. 
When $d=2k$ is even, the eigenvalue $\lambda_k$ has order $2$ and 
\[
\beta(d)=k\geq \frac{3(d-1)}{7}=\frac{6k-3}{7} \/.
\]
Now we assume that $d=2k+1$ is odd. Let $e\geq 1$ and $P=p^e$ be the largest prime power  that divides $d$. Then  $\gcd (\frac{P-1}{2}, P)=1$ and hence $\lambda_q$ has order $P$ for $q=\frac{d(P-1)}{2P}$. 
Since $d-q=\frac{d(P+1)}{2P} > q$, 
by definition we have  $\beta(d)\ge \frac{d(P-1)}{2P}$.
We need to show that $\frac{d(P-1)}{2P}\geq \frac{3(d-1)}{7}$, which is equivalent to $6P\geq d(7-P)$. When $P\geq 7$ this is obvious. 
When $P=3$ we have  $d=P=3$ and when $P=5$ we have 
$d=5$ or $d=15$. The desired inequality  $6P\geq d(7-P)$ holds in all  three cases.

\bibliographystyle{plain}
\bibliography{free}

\end{document}